\documentclass[11pt]{amsart}

\usepackage{amssymb}
\usepackage{amscd}
\usepackage{amsthm}
\usepackage{latexsym}
\allowdisplaybreaks[3]
\usepackage{graphicx}
\usepackage[british]{babel}
\usepackage[all]{xy}
\usepackage{color}
\usepackage{mathabx}
\usepackage{calrsfs}
\usepackage{booktabs}
\usepackage{enumitem}
\usepackage{chngcntr}
\usepackage{hyperref}
\usepackage{tikz-cd}

\usepackage[margin=2.5cm]{geometry}

\theoremstyle{plain}
\newtheorem{theorem}[subsection]{Theorem}
\newtheorem{lemma}[subsection]{Lemma}
\newtheorem{proposition}[subsection]{Proposition}
\newtheorem{corollary}[subsection]{Corollary}

\theoremstyle{definition}

\theoremstyle{remark}

\newtheorem{examples}[subsection]{Examples}
\newtheorem{remark}[subsection]{Remark}

\let\Join\relax
\DeclareMathOperator*{\Join}{\bigvee}
\newcommand{\meet}{\wedge}

\newcommand{\iso}{\cong}
\newcommand{\eqv}{\simeq}

\newcommand{\cat}[1]{\ensuremath{\mathsf{#1}}}
\newcommand{\A}{\cat A}

\newcommand{\C}{\cat C}
\newcommand{\D}{\cat D}

\newcommand{\catname}[1]{\ensuremath{\mathsf{#1}}}
\newcommand{\Ord}{\catname{Ord}}
\newcommand{\Pos}{\catname{Pos}}
\newcommand{\Rel}{\catname{Rel}}
\newcommand{\Set}{\catname{Set}}
\newcommand{\Top}{\catname{Top}}

\newcommand{\opname}[1]{\ensuremath{\mathsf{#1}}}
\newcommand{\Fam}{\opname{Fam}}
\newcommand{\Alg}{\opname{Alg}}
\newcommand{\op}{\opname{op}}

\newcommand{\FamX}{\Fam(X)}

\counterwithin*{equation}{section}

\begin{document}

\title{Effective descent morphisms of ordered families}

\author[M. M. Clementino]{Maria Manuel Clementino}
\address{University of Coimbra, CMUC, Department of Mathematics, 3000-143
Coimbra, Portugal}
\email{mmc@mat.uc.pt}

\author[R. Prezado]{Rui Prezado}
\thanks{The authors acknowledge partial financial support by {\it Centro de
Matemática da Universidade de Coimbra} (CMUC), funded by the Portuguese
Government through FCT/MCTES, DOI 10.54499/UIDB/00324/2020.}
\address{University of Coimbra, CMUC, Department of Mathematics, 3000-143
Coimbra, Portugal}
\email{ruiprezado@gmail.com}

\keywords{effective descent morphisms, stable regular epimorphisms, lax comma
categories, ordered families, free coproduct completion}

\subjclass{06A07,18A25,18A30,18B35,18D30}

\date{\today}

\begin{abstract}
  We present a characterization of effective descent morphisms in the lax
  comma category $\Ord//X$ when $X$ is a locally complete ordered set, as well
  as in the antisymmetric setting.
\end{abstract}

\maketitle


\section*{Introduction}

The role of lax comma 2-categories in \cite{CLN24}, where the authors study
properties of the lax change-of-base functor in the realm of Janelidze's
Galois theory \cite{GJ, BJ} led Lucatelli Nunes and the first named author of
this note to study the behaviour of the lax comma category $\Ord//X$ of
ordered sets over a fixed ordered set $X$, in \cite{CLN23}. Objects of \( \Ord
// X \) are ordered sets \( A \) equipped with a monotone map \( \alpha \colon
A \to X \), which assigns to each element \( a \in A \) an \(X\)-value \(
\alpha(a) \), and a morphism \( f \colon (A,\alpha) \to (B,\beta) \) is a
monotone map satisfying \( \alpha(a) \leq  \beta(f(a)) \).

In particular, a study of the effective descent morphisms in $\Ord//X$ was
carried out in \cite{CLN23}, when $X$ is a complete ordered set, locating them
between two well-known classes of monotone maps, as stated in Theorem
\ref{th:CLN}.  Subsequently, these results were refined in \cite{CJ23},
extending them to the case when $X$ is locally complete (Theorem \ref{th:CJ}).

In this note, we obtain a complete characterization of the effective descent
morphisms in $\Ord//X$ when $X$ is locally complete, that is, $\downarrow x$
is complete for every $x\in X$. This is accomplished by reducing the problem
to the study of effective descent morphisms in \( \Ord \) -- which were
characterized in \cite{JS} -- and in \( \FamX \) -- which were characterized
by the second named author in \cite{Prez}.

We begin by recalling the necessary descent theoretical background, and by
giving an overview of previously obtained results on effective descent
morphisms in \( \Ord // X \) in the prequels \cite{CLN23, CJ23}.

In particular, it is well-understood that \( \Ord // X \to \Ord \) preserves
effective descent morphisms when \( X \) has a bottom element. Our main
observation is that we can complete the characterization via effective descent
conditions on morphisms in the category \( \Fam(X) \). Thus, we recount the
relevant details about such morphisms from \cite{Prez}, framed in our context.
We also revisit the characterization of stable regular epimorphisms in \( \Ord
// X \) from \cite{CLN23} from the perspective of the work carried out in
\cite{Prez}.

Then, we state and prove our main result (Theorem \ref{th:char}), where we
characterize the effective descent morphisms in \( \Ord // X \) when \( X \)
has a bottom element and is locally complete.

We conclude the paper by observing that the equivalence between \( \Ord // X
\) and \( \prod_{i \in I} \Ord // X_i \), where \( X_i \) are the connected
components of \(X\), allows us to obtain our results without making use of a
bottom element of \( X \).  Indeed, with \(X\) locally complete, each \( X_i
\) has a bottom element and is locally complete as well. Therefore, the descent
results obtained for \( \Ord // X_i \) translate smoothly to \( \Ord // X \).

In the Appendix, we briefly explain how our results can be obtained in the
antisymmetric setting.

\subsection*{Acknowledgements:}

We are grateful to G.~Janelidze for several intersting suggestions, in
particular for the observation of Lemma \ref{lem:gj}, that extends the
characterization of Theorem \ref{th:char} to the setting where \(X\) does not
necessarily have a bottom element.

\subsection*{Declarations:}

No potential competing interest was reported by the authors.

\section{State-of-the-art}

In a category $\A$ with pullbacks, any morphism $p\colon A\to B$ induces a
functor \( p^* \colon \A / B \to \A / A \), by taking pullbacks along \( p \).
This functor has a left adjoint \( p_! \), and this induces a monad \( T^p \),
so we may consider the factorization of \( p^* \) through the category of \(
T^p \)-algebras (the \textit{Eilenberg-Moore} factorization):
\begin{equation}
  \label{eq:mnd.desc}
  \begin{tikzcd}[row sep=large]
    \A / B \ar[rd,"K^p",swap]
           \ar[rr,"p^*"]
      && \A / A \\
    & T^p\text{-}\Alg \ar[ru]
  \end{tikzcd}
\end{equation}
By the Bénabou-Roubaud theorem \cite{BR70}, the factorization
\eqref{eq:mnd.desc} coincides with the \textit{descent factorization}
\cite{JT97, Luc21} of \( p \) -- a result which allows the aptly called 
\textit{monadic description} of descent \cite{JT94}.

We say that
\begin{itemize}[label=--]
  \item
    \( p \) is a descent morphism if \( K^p \) is fully faithful,
  \item
    \( p \) is an effective descent morphism if \( K^p \) is an equivalence.
\end{itemize}

In a category \( \A \) with finite limits, the descent morphisms are exactly
the (pullback-)stable regular epimorphisms, which coincide with the effective
descent morphisms when $\A$ is Barr-exact or locally cartesian closed (see
\cite{JST} for details).

However, in an arbitrary category $\A$ with pullbacks, the identification of
effective descent morphisms may be quite challenging -- a notorious example is
the characterization of effective descent morphisms in the category \( \Top \)
of topological spaces \cite{RT94, CH02}.

A fruitful strategy to understand effective descent morphisms in an arbitrary
category \( \A \) with pullbacks is to find a category \( \D \) with pullbacks
for which the effective descent morphisms are well-understood, and a suitable
embedding \( F \colon \A \to \D \). Then, we may apply the following classical
result:

\begin{theorem}
  \label{th:obstruction}
  Let $\A$ and $\D$ be categories with pullbacks, and $F\colon\A\to\D$ a fully
  faithful, pullback preserving functor. If $f\colon A\to B$ is a morphism in
  $\A$ such that $F(f)$ is effective for descent in $\D$, then the following
  conditions are equivalent:
  \begin{enumerate}[label=(\roman*)]
    \item
      $f$ is an effective descent morphism in $\A$;
    \item
      for every pullback diagram of the form
      \begin{equation*}
        \begin{tikzcd}
          F(C) \ar[d] \ar[r]
            & E \ar[d,"g"] \\
          F(A) \ar[r,"F(f)",swap]
            & F(B)
        \end{tikzcd}
      \end{equation*}
      we have \( E \iso F(D) \) for some \( D \) in \( \A \).
  \end{enumerate}
\end{theorem}

This technique was used in \cite{JS} by G. Janelidze and M. Sobral to obtain 
the characterization of effective descent morphisms in the category $\Ord$ of
\textit{ordered sets} (that is, sets with a reflexive and transitive relation)
and \textit{monotone maps}:

\begin{theorem}[\cite{JS}]
  Given a morphism $f\colon A\to B$ in $\Ord$:
  \begin{enumerate}
    \item
      $f$ is a descent morphism, or, equivalently, a stable regular
      epimorphism, if:
      \begin{equation*}
        \forall \, b_0\leq b_1\mbox{ in }B,\;\;
        \exists a_0\leq a_1\mbox{ in }A: \quad f(a_0)=b_0,\;f(a_1)=b_1;
      \end{equation*}
    \item
      $f$ is an effective descent morphism if:
      \begin{equation*}
        \forall b_0\leq b_1\leq b_2\mbox{ in }B,\;\;
        \exists a_0\leq a_1\leq a_2\mbox{ in }A: \quad 
        f(a_0)=b_0,\;f(a_1)=b_1,\;f(a_2)=b_2.
      \end{equation*}
  \end{enumerate}
\end{theorem}

Moreover, Theorem \ref{th:obstruction} is also used in \cite{CLN23} and
\cite{CJ23} to study the effective descent morphisms in \( \Ord // X \). This
result is also featured in the present note.

We note that, while the characterizations of Theorem \ref{th:obstruction}
extend naturally to the comma categories $\Ord/X$ via the equivalence
\begin{equation*}
  (\Ord/X)/(B,\beta) \eqv \Ord/B,
\end{equation*}
this is not the case for the lax comma category $\Ord//X$, of which \( \Ord/X
\) is a wide subcategory (i.e. with the same objects but fewer morphisms).

In \cite{CLN23}, the authors make use of Theorem \ref{th:obstruction} and of
the fact that every monotone map $\alpha \colon A\to X$ induces naturally a
functor $\Pi(A,\alpha) \colon X^\op \to \Ord$, so that a morphism $ f \colon
(A,\alpha) \to (B,\beta)$ induces a natural transformation $\Pi(A,\alpha) \to
\Pi(B,\beta)$. Indeed:

\begin{itemize}[label=--]
  \item
    for a complete ordered set $X$, one defines an embedding
    \begin{equation*}
      \xymatrix{\Ord//X\ar[rr]^-\Pi&&[X^\op,\Ord]}
    \end{equation*}
    with $\Pi(A,\alpha)(x)=\{a\in A\,;\,x\leq\alpha(a)\}$ and $\Pi(f)$ given
    by the (co)restriction of $f$ to $\Pi(A,\alpha)(x)\to \Pi(B,\beta)(x)$:
    from $\alpha\leq \beta f$ it follows that if $x\leq \alpha(a)$ then
    $x\leq\beta(f(a))$;
  \item
    in $[X^\op,\Ord]$ a natural transformation $\eta\colon F\to G$ is
    effective for descent if and only if it is pointwise effective for
    descent, that is: for every $x\in X$, the monotone map \linebreak
    $\eta_x\colon F(x)\to G(x)$ is effective for descent in $\Ord$.
\end{itemize}

\begin{theorem}[\cite{CLN23}]
  \label{th:CLN}
  Let $X$ be a complete ordered set. Given a morphism
  $f\colon(A,\alpha)\to(B,\beta)$ in $\Ord//X$, consider the following
  conditions:
  \begin{enumerate}
    \item
      $f\colon A\to B$ and all $f_x\colon A_x\to B_x$ are effective descent
      morphisms in $\Ord$;
    \item
      $f\colon (A,\alpha)\to(B,\beta)$ is effective for descent in $\Ord//X$;
    \item
      $f\colon A\to B$ is effective for descent in $\Ord$.
  \end{enumerate}
  Then \emph{(1) $\Rightarrow$ (2) $\Rightarrow$ (3)}.
\end{theorem}

Subsequently, in \cite{CJ23} the authors use the fact that every monotone map
$\alpha\colon A\to X$ naturally defines a family $(A_x)_{x\in X}$ of subsets
of $A$ such that $A_x\subseteq A_{x'}$ whenever $x'\leq x$, and that every
monotone map $f\colon(A,\alpha)\to(B,\beta)$ satisfies $f(A_x)\subseteq B_x$
for each $x\in X$. Considering the category $\C$ having
\begin{itemize}[label=--]
  \item
    as objects, pairs $(A,(A_x)_{x\in X})$, where $A$ is an ordered set and
    $(A_x)_{x\in X}$ is a family of subsets of $A$ such that $A_x\subseteq
    A_{x'}$ whenever $x'\leq x$, 
  \item
    and as morphisms $f\colon(A,(A_x))\to(B,(B_x))$, monotone maps $f\colon
    A\to B$ such that $f(A_x)\subseteq B_x$ for each $x\in X$,
\end{itemize}
one can apply Theorem \ref{th:obstruction} based on the following facts:
\begin{itemize}[label=--]
  \item[--]
    the functor
    \[\xymatrix{\Ord//X\ar[rr]^-F&&\C,}\]
    defined by $F(A,\alpha)=(A,(A_x=\{a\in A,\;\,x\leq \alpha(a)\})_x)$ and
    $F(f)=f$, is fully faithful and preserves pullbacks;
  \item[--]
    a morphism $f\colon (A,(A_x)_x)\to(B,(B_x)_x)$ is effective for descent in
    \C\ if and only if $f\colon A\to B$ and $f_x\colon A_x\to B_x$, for all \(
    x \in X \), are surjective.
\end{itemize}

\begin{theorem}[\cite{CJ23}]
  \label{th:CJ}
  Let $X$ be a locally complete ordered set with a bottom element. For a
  morphism $f\colon(A,\alpha)\to(B,\beta)$ in $\Ord//X$, consider the
  following conditions:
  \begin{enumerate}
    \item
      In $\Ord$, $f\colon A\to B$ is effective for descent, and $f_x\colon
      A_x\to B_x$ is a descent morphism for all \( x \in X \);
    \item
      $f\colon(A,\alpha)\to(B,\beta)$ is effective for descent in $\Ord//X$.
  \end{enumerate}
  Then \emph{(1) $\Rightarrow$ (2)}. If, in addition, for each $x\in X$ every
  subset of $\downarrow \hspace*{-1mm}x$ has a largest element, then \emph{(1)
  $\Leftrightarrow$ (2)}.
\end{theorem}

Theorem \ref{th:CJ} gives us, for a locally complete ordered set $X$, a
sufficient condition for $f$ to be effective for descent in $\Ord//X$ which is
not necessary in general, as we show in the sequel. Indeed, in order to apply
Theorem \ref{th:obstruction}, we must start with a morphism whose $F$-image is
an effective descent morphism in \C, hence $f$ and all $f_x$ are a priori
surjective, and this condition is not fulfilled by all effective descent
morphisms in $\Ord//X$, as we show in Example \ref{ex:int}.

\section{Familial descent}

One of the main insights behind our main result, Theorem \ref{th:char}, is
that we can reduce the study of effective descent morphisms (respectively,
stable regular epimorphisms) in \( \Ord // X \) to the study of effective
descent morphisms (respectively, stable regular epimorphisms) in \( \FamX \)
and \( \Ord \). This latter problem in \( \FamX \) has been considered before
in \cite[Lemma~4.4]{Prez} (see also \cite[Lemma~3.17]{PrezTh}), from which we
proceed to recall the relevant details.

For a fixed ordered set \(X\), we denote by \( \FamX \) the category of
set-indexed \textit{families of elements} in \(X\). It consists of:
\begin{itemize}[label=--]
  \item
    Objects: families \( (\alpha_j)_{j \in J} \) of elements \( \alpha_j \in X \)
    indexed by a set \( J \),
  \item
    Morphisms \( (\alpha_j)_{j \in J} \to (\beta_k)_{k \in K} \): a function
    \( f \colon J \to K \) such that \( \alpha_j \leq \beta_{f(j)} \) for all \(
    j \in J \).
\end{itemize}

We will assume that \( X \) \textit{locally has binary meets}, that is,
$\downarrow x$ has binary meets for all \(x\). When \(X\) is seen as a thin
category, this condition is equivalent to saying that \(X\) has pullbacks.
Thus, it follows that \( \FamX \) is a category with pullbacks (see, e.g.
\cite[Sections~6.2, 6.3]{BJ}).

We also recall that an ordered set \(X\) with finite meets is said to
be \textit{cartesian closed} if there is an assignment \( (y,z) \mapsto z^y
\), which satisfies 
\begin{equation*}
  x \meet y \leq z \iff x \leq z^y
\end{equation*}
for every \(x \in X \). When \(X\) is complete and the underlying order is
antisymmetric, this is equivalent to \(X\) being a \textit{frame}.
Likewise, an ordered set \( X \) with locally binary meets is said to be
\textit{locally cartesian closed} if \( \downarrow x \) is cartesian closed
for all \( x \in X \). 

While the results of \cite{Prez, PrezTh} study (effective) descent morphisms
in \( \Fam(X) \) when \(X\) has a top element -- due to the pertinence of the
work carried out within -- the results plainly extend to the setting where
\(X\) does not admit a top element.

\begin{lemma}[{\cite[Lemma~4.4]{Prez}, \cite[Lemma~3.17]{PrezTh}}]
  \label{lem:fam.desc}
  Let \( f \colon (\alpha_j)_{j \in J} \to (\beta_k)_{k \in K} \) be a
  morphism in \( \FamX \).
  \begin{enumerate}
    \item
      \label{enum:desc}
      \(f\) is a descent morphism if and only if
      \begin{equation}
        \label{eq:join.distr}
        \forall k \in K, \;\; \forall w \leq \beta_k, \qquad
        w \iso \Join_{f(j)=k} w \meet \alpha_j
      \end{equation} 
    \item
      \label{enum:char}
      If \( X \) is locally complete, then \(f\) is an effective descent
      morphism if and only if \(f\) is a descent morphism and for every family
      \( (\sigma_j)_{j \in J} \leq (\alpha_j)_{j \in J} \)
      satisfying\footnote{Such families satisfying \eqref{eq:descent.data} are
      said to be \textit{descent data} for $f$.}
      \begin{equation}
        \label{eq:descent.data}
        \forall\,k\in K, \;\; \forall\,i,\,j \in f^{-1}(k), \quad
        \sigma_j \meet \alpha_i \iso \alpha_j \meet \sigma_i,
      \end{equation}
      we have
      \begin{equation}
        \label{eq:effectiveness}
        \forall\,k\in K, \;\; \forall\,j \in f^{-1}(k), \quad
        \alpha_j \meet \Join_{i \in f^{-1}(k)} \sigma_i \iso \sigma_j.
      \end{equation} 
    \item
      \label{enum:lcc}
      If \( X \) is locally complete and locally cartesian closed, then \( f
      \) is an effective descent morphism if and only if \( f \) is a descent
      morphism.
  \end{enumerate}
\end{lemma}

\begin{proof}
  We first verify that having a top element is redundant. By \cite[Theorem
  3.4(a)]{JST}, we note that \( f \) is a descent morphism in \( \Fam(X) \) if
  and only if it is a stable regular epimorphism in
  \begin{equation*}
    \Fam(X) / (\beta_k)_{k \in K}
      \eqv \prod_{k \in K} \Fam(X)/\beta_k
      \eqv \prod_{k \in K} \Fam(X/\beta_k),
  \end{equation*}
  which is the case if and only if \eqref{eq:join.distr} holds.

  We note that \eqref{enum:desc} follows directly from
  \cite[Lemma~3.17(d)]{PrezTh}.

  To conclude \eqref{enum:char}, we use \cite[Lemma~3.17(e)]{PrezTh}, noting
  that, in lextensive categories (such as \( \Fam(X) \)), (effective) descent
  morphisms are closed under coproducts -- the product of (pre)monadic
  functors is itself (pre)monadic -- so it is enough to confirm that \(
  f\colon (\alpha_j)_{j \in f^{-1}(k)} \to \beta_k \) is an (effective)
  descent morphism for all \(k \in K \) (see \cite[Lemma~3.5]{PrezTh}).

  Finally, we note that \eqref{enum:lcc} follows directly from
  \eqref{enum:char} by distributivity.
\end{proof}

Lemma~\ref{lem:fam.desc} on its own already allows us to smoothly extend the
characterization of stable regular epimorphisms obtained in \cite[Lemma 3.1,
Proposition 3.2]{CLN23} for \(X\) complete and cartesian closed to our
context.

\begin{proposition}
  \label{lem:stabregepi}
  Let \(X\) be a locally complete ordered set with a bottom element, and let
  \( f \colon (A,\alpha) \to (B,\beta) \) be a morphism in \( \Ord // X \).
  \begin{enumerate}
     \item 
     \label{enum:regepi}    
      \(f\) is a regular epimorphism in \( \Ord // X \) if and only if it is a
      regular epimorphism in \( \Ord \) and
      \begin{equation*}
        \forall b \in B, \quad 
          \beta(b) \iso \Join_{f(a)\leq b} \alpha(a).
      \end{equation*}
    \item
      \label{enum:stabregepi}
      \(f\) is a stable regular epimorphism in \( \Ord // X \) if and only if
      it is a stable regular epimorphism in \( \Ord \) and
      \begin{equation}
        \label{eq:stablereg}
        \forall b\in B, \;\;\forall w\leq\beta(b), \quad 
          w \iso \Join_{f(a)=b} w \meet \alpha(a).
      \end{equation} 
  \end{enumerate}
\end{proposition}

\begin{proof}
  We note that \eqref{enum:regepi} is precisely \cite[Lemma 3.1]{CLN23}, so we
  focus on \eqref{enum:stabregepi}.

  If \( f \) is a stable regular epimorphism in \( \Ord // X \), then, for each
  \( b \in B \) and \( w \leq \beta(b) \), we consider the pullback diagram
  \begin{equation*}
    \begin{tikzcd}
      \Big( f^{-1}(b), \big(a \mapsto w \meet \alpha(a)\big) \Big)
        \ar[r,"u"] \ar[d]
        & (b, w) \ar[d] \\
      (A, \alpha) \ar[r,"f",swap]
        & (B,\beta)
    \end{tikzcd}
  \end{equation*}
  so that \( u \) is a regular epimorphism in \( \Ord // X \), which entails
  \eqref{eq:stablereg}, as desired.

  Conversely, if \eqref{eq:stablereg} holds, then for any pullback diagram
  \begin{equation*}
    \begin{tikzcd}
      \Big(A \times_B C,
           \big((a,c) \mapsto \gamma(c) \meet \alpha(a) \big)\Big)
        \ar[r,"\pi_2"] \ar[d]
        & (C, \gamma) \ar[d,"g"] \\
      (A, \alpha) \ar[r,"f",swap]
        & (B,\beta)
    \end{tikzcd}
  \end{equation*}
  we claim that \( \pi_2 \) is a (stable) regular epimorphism. Indeed, for
  each \( c \in C \) we have \( \gamma(c) \leq \beta(g(c)) \), so from
  \eqref{eq:stablereg} we deduce that
  \begin{equation*}
    \gamma(c) \iso \Join_{f(a)=g(c)} \gamma(c) \meet \alpha(a)
              \iso \Join_{\substack{f(a')=g(c') \\ c' \leq c}}
                  \gamma(c') \meet \alpha(a')
  \end{equation*}
  which indeed confirms that \( \pi_2 \) is a regular epimorphism.
\end{proof}

\begin{remark}
  \label{rem:stabregepi.to.edm}
  We point out that condition \eqref{eq:stablereg} can be interpreted in the
  category $\FamX$ by considering the (faithful) forgetful functor
  \[\xymatrix{\Ord//X\ar[rr]&&\FamX}\]
  which maps $(A,\alpha)$ to the family $(\alpha(a))_{a\in A}$. Thus, by Lemma
  \ref{lem:fam.desc}, condition \eqref{enum:stabregepi} can be restated as
  follows: \textit{$f \colon (A,\alpha) \to (B,\beta)$ is a stable regular
  epimorphism in $\Ord//X$ if and only if the underlying morphisms in \( \Ord
  \) and \( \FamX \) are stable regular epimorphisms.}

  In fact, we can say more: since \( X \) is assumed to be locally complete,
  \( \Ord // X \to \Fam(X) \) preserves effective descent morphisms, by
  arguments analogous to those of \cite[Section~2]{CJ23}. Therefore, when \( X
  \) is locally complete and has a bottom element, we conclude that both
  forgetful functors
  \begin{equation}
    \label{eq:forgetful}
    \begin{tikzcd}
      & \Ord//X \ar[rd] \ar[ld] \\
      \Ord && \FamX
    \end{tikzcd}
  \end{equation}
  preserve effective descent morphisms.
\end{remark}

\section{The characterization}

Having reviewed the necessary details, we may proceed to prove our main
result:

\begin{theorem}
  \label{th:char}
  Let $X$ be a locally complete ordered set with a bottom element. A morphism
  $f\colon (A,\alpha) \to (B,\beta)$ is effective for descent in $\Ord//X$ if
  and only if
  \begin{enumerate}[label=(\arabic*)]
    \item
      \label{enum:left}
      $f\colon A\to B$ is effective for descent in $\Ord$; that is
      \begin{equation*}
        \forall b_0\leq b_1\leq b_2\text{ in }B,\;\;\exists a_0\leq a_1\leq a_2
        \text{ in }A\colon\quad f(a_0)=b_0,\;f(a_1)=b_1,\;f(a_2)=b_2.
      \end{equation*}
    \item
      \label{enum:right}
      we have
      \begin{equation*}
        \forall\: b_0\leq b_1,\:\forall\: w \leq \beta(b_0), \quad
        w \iso \Join_{\substack{a_0 \leq a_1 \\ f(a_i)=b_i}} w \meet \alpha(a_0).
      \end{equation*}
    \item
      \label{enum:center}
      for every family \( (\sigma(a))_{a \in A} \leq (\alpha(a))_{a \in A} \)
      satisfying 
      \begin{equation}
        \label{eq:descprop}
        \forall \, b \in B, \;\;
        \forall \, a, a' \in f^{-1}(b), \quad
        \sigma(a') \meet \alpha(a) \iso \alpha(a') \meet \sigma(a),
      \end{equation}
      we have
      \begin{equation*}
        \forall \, b \in B, \;\;
        \forall \, a' \in f^{-1}(b), \quad
        \alpha(a') \meet \Join_{a \in f^{-1}(b)} \sigma(a) \iso \sigma(a').
      \end{equation*}
  \end{enumerate}
\end{theorem}

To prove this result, it is natural to consider the (pseudo)pullback diagram
below (see \cite{JS93}, noting that \( \FamX \to \Set \) is an
(iso)fibration):
\begin{equation*}
  \begin{tikzcd}
      \Ord\times_\Set \FamX
        \ar[d,"\rho_1",swap]
        \ar[r,"\rho_2"]
      & \FamX \ar[d] \\
      \Ord \ar[r]
      & \Set
  \end{tikzcd}
\end{equation*}
as well as the functor \( \Ord // X \to \Ord \times_\Set \FamX \) induced by
the forgetful functors~\eqref{eq:forgetful}.

Via \cite[Corollary~9.6]{Luc} and a suitable adjustment of
\cite[Corollary~2.6]{CJ23}, we obtain:

\begin{lemma}
  \label{lem:inbetween}
  Let $X$ be a locally complete ordered set with a bottom element.  A morphism
  $f$ is effective for descent in $\Ord\times_\Set\FamX$ if and only if:
  \begin{enumerate}
    \item
      $\rho_1(f)$ is effective for descent in \Ord.
    \item
      $\rho_2(f)$ is effective for descent in $\FamX$.
  \end{enumerate}
\end{lemma}

Now, the fully faithful and pullback preserving functor
\begin{equation*}
  \xymatrix{\Ord//X\ar[rr]^-U&&\Ord\times_\Set\FamX}
\end{equation*}
and Theorem \ref{th:obstruction} give us the tools to characterize effective
descent morphisms in $\Ord//X$. Before proceeding to the proof, we recall
that the objects $\Ord\times_\Set\FamX$ consist of pairs $(C,(\chi_c)_{c\in C})$
where $C$ is an ordered set and $(\chi_c)_{c\in C}$ is a family of elements of
$X$, that is, a map $\chi\colon C\to X$. Such a pair is in the (essential)
image of \( U \) if and only if \( \chi \) is monotone.

\begin{proof}[Proof of Theorem \ref{th:char}]
  Let \( f \colon (A,\alpha) \to (B, \beta) \) be a morphism in \( \Ord // X
  \) satisfying conditions \ref{enum:left}--\ref{enum:center}.

  Given that condition \ref{enum:right} holds, we note that, for all \( b_0
  \leq b_1 \) in \( B \) and all \( w \leq \beta(b_0) \) in \(X\), we have
  \begin{equation*}
    w \iso \Join_{\substack{a_0 \leq a_1 \\ f(a_i)=b_i}} w \meet \alpha(a_0)
      \leq \Join_{a_0 \in f^{-1}(b_0)} w \meet \alpha(a_0)
      \leq w,
  \end{equation*}
  hence, together with condition \ref{enum:center}, we conclude that \(
  \rho_2(U(f)) \) is an effective descent morphism in \( \FamX \) by Lemma
  \ref{lem:fam.desc}.  Thus, if condition \ref{enum:left} also holds, \(
  \rho_1(U(f)) \) is an effective descent morphism in \( \Ord \), so \( U(f)
  \) is an effective descent morphism in \( \Ord \times_\Set \FamX \).

  Now, we apply Theorem \ref{th:obstruction}: if we have a pullback diagram
  \begin{equation*}
    \xymatrix{U(D,\delta)\ar[r]\ar[d]&(C,(\chi_c)_{c\in C})\ar[d]^g\\
    U(A,\alpha)\ar[r]_{U(f)}&U(B,\beta)}
  \end{equation*}
  we want to show that $\chi\colon C\to X$ is monotone. Let $c_0\leq c_1\in C$
  and let $b_i=g(c_i)$ for $i=0,1$. Then $\chi(c_0)\leq \beta(b_0)$ and
  therefore, by condition \ref{enum:right},
  \begin{equation*}
    \chi(c_0) 
      \iso \Join_{\substack{a_0 \leq a_1 \\ f(a_i)=b_i}} 
              \chi(c_0) \meet \alpha(a_0)
      =    \Join_{\substack{a_0 \leq a_1 \\ f(a_i)=b_i}}
              \delta(a_0,c_0)
      \leq \Join_{a\in f^{-1}(b_1)}
              \delta(a,c_1)\leq \chi(c_1),
  \end{equation*}
  as desired.

  Conversely, if \( f \) is an effective descent morphism in $\Ord//X$, then
  by Remark \ref{rem:stabregepi.to.edm} it follows that both \( \rho_1(U(f))
  \) and \( \rho_2(U(f)) \) are effective descent morphisms in \( \Ord \) and
  \( \FamX \), respectively, from which we conclude that \( U(f) \) is an
  effective descent morphism in $\Ord \times_\Set \Fam(X)$ (by Lemma
  \ref{lem:inbetween}), and that condition \ref{enum:left} holds.

  To prove condition \ref{enum:right}, we apply Theorem \ref{th:obstruction}
  again: we let \( b_0 \leq b_1 \) and \( w \leq \beta(b_0) \), and we
  consider the pair \( \big( \{b_0,b_1\}, \big(\chi_{b_0},\chi_{b_1})\big) \),
  where
  \begin{equation*}
    \chi_{b_0}=w,\qquad
    \qquad
    \chi_{b_1}=\Join_{\substack{a_0 \leq a_1 \\ f(a_i)=b_i}} w\meet \alpha(a_0).
  \end{equation*}
  We also let
  \begin{equation*}
    g \colon \big(\{b_0, b_1\},(\chi_{b_0},\chi_{b_1})\big) \to (B,\beta)
  \end{equation*}
  be the inclusion. Taking the pullback of $U(f)$ along $g$, we obtain
  \begin{equation*}
    \begin{tikzcd}
      \big(D,(\xi_d)_{d\in D}\big)
        \ar[r] \ar[d]
        & \big(\{b_0,b_1\},(\chi_{b_0},\chi_{b_1})\big)
          \ar[d,"g"] \\
      U(A,\alpha)
        \ar[r,"U(f)",swap]
        & U(B,\beta),
    \end{tikzcd}
  \end{equation*}
  where $D=\{(a,b_i)\;|\;f(a)=b_i, \,i=0,1\}$, and \( \xi \) is given by
  \begin{itemize}[label=--]
    \item
      \( \xi_{(a,b_0)} = \alpha(a) \meet w \) for each \( a \in A \)
      such that \( f(a) = b_0 \), and
    \item
      \( \xi_{(a,b_1)} =\alpha(a)\meet \chi_{b_1} \) for each \( a \in A
      \) such that \( f(a) = b_1 \).
  \end{itemize}
  Hence, if $(a,b_0)\leq (a',b_1)$, then \( a \leq a' \) and \( f(a') =
  b_1 \). It follows that
  \begin{equation*}
    \alpha(a) \meet w \leq \chi_{b_1} \qquad \text{and} \qquad
    \alpha(a) \meet w \leq \alpha(a'),
  \end{equation*}
  and therefore \( \xi_{a,b_0} \leq \xi_{a',b_1} \). Monotonicity of \(
  \alpha \) covers the remaining cases (when \((a,b_i) \leq (a',b_i)\) for \(
  i=0,1 \)), and thereby we conclude that \( \xi \) is monotone. Thus, \( \chi
  \) must be monotone as well, so that \( \chi_{b_0} \iso \chi_{b_1} \),
  confirming that condition \ref{enum:right} holds. 
\end{proof}

\begin{corollary}
  \label{cor:char}
  Let \(X\) be a locally cartesian closed, locally complete ordered set with a
  bottom element. A morphism \( f \colon (A,\alpha) \to (B,\beta) \) in \(
  \Ord // X \) is an effective descent morphism if and only if
  \begin{enumerate}[label=(\arabic*)]
    \item
      \label{enum:left.2}
      \( f \colon A \to B \) is effective for descent in \( \Ord \); that is
      \begin{equation*}
        \forall b_0\leq b_1\leq b_2\text{ in }B\;\;\exists a_0\leq a_1\leq a_2
        \text{ in }A\colon\quad f(a_0)=b_0,\;f(a_1)=b_1,\;f(a_2)=b_2.
      \end{equation*}
    \item
      \label{enum:right.2}
      we have
      \begin{equation*}
        \forall \: b_0 \leq b_1, \quad
        \beta(b_0) \iso \Join_{\substack{a_0 \leq a_1 \\ f(a_i)=b_i}} \alpha(a_0),
      \end{equation*}
  \end{enumerate}
\end{corollary}

\begin{proof}
  Since \(X\) is locally cartesian closed, local meets distribute over local
  joins, hence we have
  \begin{equation*}
    w \iso w \meet \beta(b_0) 
      \iso w \meet \Join_{\substack{a_0\leq a_1 \\ f(a_i)=b_i}} \alpha(a_0)
      \iso \Join_{\substack{a_0 \leq a_1 \\ f(a_i)=b_i}} w \meet \alpha(a_0)
  \end{equation*}
  for all \( w \leq \beta(b_0) \).  Moreover, given a family \( (\sigma_a)_{a
  \in A} \leq (\alpha_a)_{a \in A} \) satisfying \eqref{eq:descprop}, we have
  \begin{equation*}
    \alpha(a') \meet \Join_{a \in f^{-1}(b)} \sigma(a) 
      \iso \Join_{a \in f^{-1}(b)} \alpha(a') \meet \sigma(a) 
      \iso \Join_{a \in f^{-1}(b)} \sigma(a') \meet \alpha(a) 
      \iso \sigma(a')
  \end{equation*}
  for each \( b \in B \) and each \( a' \in f^{-1}(b) \).  Now, we may apply
  Theorem \ref{th:char}.
\end{proof}

\begin{examples}
  \label{ex:int}
  Let $X$ be the interval $[0,1]$ with the usual order -- we observe that
  \(X\) is a cartesian closed, complete ordered set.

  Let $A=\{(x,y)\in X^2\,;\,y<x\text{ or }y=x=0\}$, and write \( \alpha=\pi_2
  \), \( f = \pi_1 \) for the projections.

  \begin{enumerate}[label=(\Roman*)]
    \item
      If we equip \( A \) with the product order, then both \( \alpha \) and
      \(f\) are monotone, so that we have a morphism 
      \begin{equation}
        \label{eq:key.morph}
        \begin{tikzcd}[row sep=large]
          A \ar[rd,"\alpha"{name=A},swap]
            \ar[rr,shift right=6.5, phantom,"\leq"{anchor=center}]
            \ar[rr,"f"{name=B}] 
          && X \ar[ld,equal]\\
          & X
        \end{tikzcd}
      \end{equation}
      in \( \Ord // X \) -- indeed, we note that $\alpha(x,y)=y\leq
      f(x,y)=x$. Moreover,
      \begin{itemize}[label=--]
        \item
          $f$ is effective for descent in $\Ord$:
          \begin{equation}
            \label{eq:f.edm.ord}
            \forall x_0\leq x_1\leq x_2\text{ in }[0,1]\quad
            \exists (x_0,0)\leq (x_1,0)\leq (x_2,0)\text{ in }A\;:\;
            f(x_i,0)=x_i.
          \end{equation}
        \item
          If $0=x\leq x'$ then $(x,0)\leq (x',0)$ in $A$ and $0 =
          \alpha(x,0)$; if $0 < x \leq x'$, then, for all $0 \leq y < x$,
          $(x,y)\leq(x',y)$ in $A$ and clearly $ x = \bigvee \{\alpha(x,y)
          \,|\, 0 \leq y<x\}$, 
      \end{itemize}
      and these respectively correspond to conditions \ref{enum:left.2} and
      \ref{enum:right.2} of Corollary \ref{cor:char}. We conclude that \( f \)
      is an effective descent morphism in \( \Ord // X \). We highlight that
      not every $f_x$ is surjective for $x \in X$, so this $f$ is not under
      the conditions of \cite[Theorem 5.3]{CJ23}.
    \item
      If we consider on \( A \) the order defined by
      \begin{equation*}
        (x,y) \leq (x',y') \iff
          \quad (x,y) = (x',y') \quad \text{ or } \quad 
          x \leq x' \text{ and } y=y'=0,
      \end{equation*}
      then, once again, both \( \alpha \) and \(f\) are monotone, and \( f \)
      defines a morphism \eqref{eq:key.morph} in \( \Ord // X \). 

      Moreover, we note that \( f \) \textit{is an effective descent morphism}
      in \( \Ord \), since \eqref{eq:f.edm.ord} still holds, and that \(f\)
      \textit{is a stable regular epimorphism} in \( \Ord // X \), because we
      have \( f^{-1}(x) = \{ (x,y) \in A \;|\; 0 \leq y < x \} \), hence
      \begin{equation*}
        \forall x \in X \quad
        x \iso \Join_{y < x} y.
      \end{equation*}

      However, $ f $ \textit{is not effective for descent} in $\Ord //X$: if
      \( 0 < x < 1 \), then \( (x,y_0) \leq (1,y_1) \) in \(A\) only if \(
      y_0=y_1=0 \), hence 
      \begin{equation*}
        \Join_{(x,y_0) \leq (1,y_1)} \alpha(x,y_0) = \alpha(x,0) = 0 < x,
      \end{equation*}
      so \( f \) does not satisfy condition \ref{enum:right.2} of
      Corollary~\ref{cor:char}.
  \end{enumerate}
\end{examples}

\section{Redundancy of the bottom element}

Thanks to an observation due to G. Janelidze (private communication), we are
able to obtain the results of Proposition \ref{lem:stabregepi} and Theorem
\ref{th:char} even if \( X \) has no bottom element.

Each ordered set \(X\) is the coproduct \( X \iso \sum_{i\in I} X_i \) of its
connected components, where \( x, x' \in X \) belong to the same connected
component if there exists a zigzag
\begin{equation*}
  x = x_0 \leq x_1 \geq x_2 \leq \ldots \geq x_{n-1} \leq x_n = x'.
\end{equation*}
where \( x_i \in X \).

Each object \( (A, \alpha \colon A \to X) \) of \( \Ord // X \) induces a
family \( (A_i, \alpha_i \colon A_i \to X_i)_{i \in I} \) of objects belonging
to each \( \Ord // X_i \), while the (co)restrictions of \( f \colon
(A,\alpha) \to (B,\beta) \) define a family of morphisms \( (f_i \colon
(A_i,\alpha_i) \to (B_i,\beta_i))_{i\in I} \) belonging to each  \( \Ord //
X_i \). This defines a functor
\begin{equation*}
  \begin{tikzcd}
    \Ord // X \ar[r]
      & \prod_{i \in I} \Ord // X_i
  \end{tikzcd}
\end{equation*}
which is easily seen to be an equivalence.

\begin{lemma}
  \label{lem:gj}
  Let \( X \) be a locally complete ordered set. Its connected components are
  locally complete ordered sets with bottom element.  Moreover, if \( X \) is
  locally cartesian closed, then so is each component.
\end{lemma}

\begin{proof}
  For \( x \in X \), we denote the bottom element of the downset \( \downarrow
  x \) by \( \bot_x \), which exists by completeness.

  It is enough to confirm that if we have \( x \leq y \) in \( X \), then \(
  \bot_x \iso \bot_y \) in \( X \). And indeed this is the case: an immediate
  calculation shows that we have a chain
  \begin{equation*}
    \bot_x \leq \bot_y \leq \bot_x \leq x \leq y,
  \end{equation*}
  confirming our statement.

  Finally, we note that the downsets of the connected components of \(X\)
  coincide with the downsets of \(X\).
\end{proof}

Lemma \ref{lem:gj}, together with the observations that \( \Ord // \sum_{i \in
I} X_i \eqv \prod_{i \in I} \Ord // X_i \) and that descent properties on
products are encoded by descent properties on the components, we conclude
that:

\begin{theorem}
  \label{thm:stabregepi.2}
  Let \(X\) be a locally complete ordered set, and let \( f \colon (A,\alpha)
  \to (B,\beta) \) be a morphism in \( \Ord // X \).
  \begin{enumerate}
     \item 
     \label{enum:regepi}    
      \(f\) is a regular epimorphism in \( \Ord // X \) if and only if it is a
      regular epimorphism in \( \Ord \) and
      \begin{equation*}
        \forall b \in B, \quad
          \beta(b) \iso \Join_{f(a)\leq b} \alpha(a).
      \end{equation*}
    \item
      \label{enum:stabregepi}
      \(f\) is a stable regular epimorphism in \( \Ord // X \) if and only if
      it is a stable regular epimorphism in \( \Ord \) and
      \begin{equation*}
        \forall b\in B,\;\;\forall w\leq\beta(b),\quad
          w \iso \Join_{f(a)=b} w \meet \alpha(a).
      \end{equation*}
  \end{enumerate}
\end{theorem}

\begin{theorem}
  \label{thm:char.2}
  Let \(X\) be a locally complete ordered set. A morphism \( f \colon (A,
  \alpha) \to (B, \beta) \) is effective for descent in \( \Ord // X \) if and
  only if
  \begin{enumerate}[label=(\arabic*)]
    \item
      $f\colon A\to B$ is effective for descent in $\Ord$; that is
      \begin{equation*}
        \forall b_0\leq b_1\leq b_2\text{ in }B\;\;\exists a_0\leq a_1\leq a_2
        \text{ in }A\colon\quad f(a_0)=b_0,\;f(a_1)=b_1,\;f(a_2)=b_2.
      \end{equation*}
    \item
      we have
      \begin{equation*}
        \forall\: b_0\leq b_1,\:\forall\: w \leq \beta(b_0), \quad
        w = \Join_{\substack{a_0 \leq a_1 \\ f(a_i)=b_i}} w \meet \alpha(a_0).
      \end{equation*}
    \item
      for every family \( (\sigma(a))_{a \in A} \leq (\alpha(a))_{a \in A} \)
      satisfying 
      \begin{equation*}
        \forall \,b\in B,\;\; \forall \, a,a' \in f^{-1}(b), \qquad
        \sigma(a') \meet \alpha(a) \iso \alpha(a') \meet \sigma(a),
      \end{equation*}
      we have
      \begin{equation*}
        \forall \,b\in B,\;\; \forall \, a' \in f^{-1}(b), \qquad
        \alpha(a') \meet \Join_{a \in A} \sigma(a) \iso \sigma(a').
      \end{equation*}
  \end{enumerate}
\end{theorem}

\begin{corollary}
  Let \( X \) be a locally complete, locally cartesian closed ordered set with
  a bottom element.  A morphism \( f \colon (A,\alpha) \to (B,\beta) \) is
  effective for descent in \( \Ord // X \) if and only if 
  \begin{enumerate}[label=(\arabic*)]
    \item
      \label{enum:left.3}
      \( f \colon A \to B \) is effective for descent in \( \Ord \); that is
      \begin{equation*}
        \forall b_0\leq b_1\leq b_2\text{ in }B\;\;\exists a_0\leq a_1\leq a_2
        \text{ in }A\colon\quad f(a_0)=b_0,\;f(a_1)=b_1,\;f(a_2)=b_2.
      \end{equation*}
    \item
      \label{enum:right.3}
      we have
      \begin{equation*}
        \forall \: b_0 \leq b_1, \quad
        \beta(b_0) \iso \Join_{\substack{a_0 \leq a_1 \\ f(a_i)=b_i}} \alpha(a_0).
      \end{equation*}
  \end{enumerate}
\end{corollary}

\appendix
\section{Antisymmetry}

Let \( \Pos \) be the full subcategory of \( \Ord \) consisting of the
ordered, antisymmetric sets (\textit{posets}). Likewise, for a poset \(X\), we
denote by \( \Pos // X \) the full subcategory of \( \Ord // X \) consisting
of those pairs \( (A,\alpha \colon A \to X) \) in \( \Ord // X \) such that \(
A \) is a poset. We observe that descent theory in \( \Pos // X \) can be
carried out just as in \( \Ord // X \). 

\begin{lemma}
  \label{lem:ap}
  The fully faithful functor \( F \colon \Pos \to \Ord \) preserves and
  reflects effective descent morphisms.
\end{lemma}

\begin{proof}
  \(J\) has a left adjoint, given by the posetal reflection of an ordered set
  (equivalence classes of isomorphic elements). 

  Let \( f \colon A \to B \) be a morphism in \( \Pos \). If \( F(f) \) is an
  effective descent morphism in \( \Ord \), we note that condition (ii) of
  Theorem \ref{th:obstruction} holds unconditionally, since any surjective
  monotone map \( F(C) \to E \) in \( \Ord \) implies that \( E \) is a poset.
  Thus, \( f \) is an effective descent morphism in \( \Pos \).

  The preservation follows by \cite[Proposition 3.2]{JS} applied to the
  embedding \( \Pos \to \Rel \).
\end{proof}

We conclude that a monotone map \( f \colon X \to Y \) is an effective descent
morphism in \( \Pos \) if and only if it is an effective descent morphism in
\( \Ord \). Using this fact, we can carry out the results of
Proposition~\ref{lem:stabregepi}, and Theorems \ref{th:char},
\ref{thm:stabregepi.2}, \ref{thm:char.2} in the setting of posets, with
minimal changes:

\begin{theorem}
  Let \( X \) be a locally complete poset, and let \( f \colon (A,\alpha) \to
  (B,\beta) \) be a morphism in \( \Pos // X \). We have that:
  \begin{enumerate}
    \item
      \label{enum:regepi.tnd}
      \(f\) is a regular epimorphism if and only if it is a regular
      epimorphism in \( \Pos \) and 
      \begin{equation*}
        \forall b \in B, \quad
          \beta(b) = \Join_{f(a)\leq b} \alpha(a).
      \end{equation*}
    \item
      \label{enum:stabregepi.tnd}
      \(f\) is a stable regular epimorphism if and only if it is a stable
      regular epimorphism in \( \Pos \) and
      \begin{equation*}
        \forall b\in B,\;\;\forall w\leq\beta(b),\quad
          w = \Join_{f(a)=b} w \meet \alpha(a).
      \end{equation*}
    \item
      \label{enum:edm.tnd}
      \(f\) is an effective descent morphism if and only if 
      \begin{itemize}[label=--]
        \item
          \( f \) is an effective descent morphism in \( \Pos \), 
        \item
          we have
          \begin{equation*}
            \forall \: b_0 \leq b_1, \: \forall w \leq \beta(b_0), \quad
            w = \Join_{\substack{a_0 \leq a_1 \\ f(a_i)=b_i}} 
                  w \meet \alpha(a_0).
          \end{equation*}
        \item
          for all \( (\sigma(a))_{a \in A} \leq (\alpha(a))_{a \in A} \)
          satisfying
          \begin{equation*}
            \forall \, b \in B, \;\;
            \forall \, a, a' \in f^{-1}(b), \quad
            \sigma(a') \meet \alpha(a) \iso \alpha(a') \meet \sigma(a),
          \end{equation*}
          we have
          \begin{equation*}
            \forall \, b \in B, \;\;
            \forall \, a' \in f^{-1}(b), \quad
            \alpha(a') \meet \Join_{a \in A} \sigma(a)
              \iso \sigma(a').
          \end{equation*}
      \end{itemize}
 \end{enumerate}
\end{theorem}

\begin{corollary}
  Let \( X \) be locally a frame -- that is, \( \downarrow x \) is a frame for
  all \( x \in X \) -- and let \( f \colon (A,\alpha) \to (B,\beta) \) be a morphism
  in \( \Pos // X \). We have that \( f \) is an effective descent morphism if
  and only if 
  \begin{enumerate}[label=(\arabic*)]
    \item
      \label{enum:left.4}
      \( f \colon A \to B \) is effective for descent in \( \Pos \); that is
      \begin{equation*}
        \forall b_0\leq b_1\leq b_2\text{ in }B,\;\;\exists a_0\leq a_1\leq a_2
        \text{ in }A\colon\quad f(a_0)=b_0,\;f(a_1)=b_1,\;f(a_2)=b_2.
      \end{equation*}
    \item
      \label{enum:right.4}
      we have
      \begin{equation*}
        \forall \: b_0 \leq b_1, \quad
        \beta(b_0) = \Join_{\substack{a_0 \leq a_1 \\ f(a_i)=b_i}} \alpha(a_0).
      \end{equation*}
  \end{enumerate}
\end{corollary}

\end{document}